\documentclass[11pt]{amsart}

\usepackage{amsmath,amstext,amsopn,amsfonts,amssymb,amsthm,graphicx,caption,subcaption,tikz}
\usetikzlibrary{matrix,arrows,decorations.pathmorphing}
\usepackage[abs]{overpic}
\usepackage[foot]{amsaddr}

\theoremstyle{definition}

\newtheorem{theorem}{Theorem}[section]
\newtheorem{proposition}[theorem]{Proposition}

\newtheorem{definition}[theorem]{Definition}
\newtheorem{corollary}[theorem]{Corollary}
\newtheorem{lemma}[theorem]{Lemma}
\newtheorem{remark}[theorem]{Remark}
\newtheorem{example}[theorem]{Example}

\newcommand{\pderiv}[2][\!]{\frac{\partial^{#1}}{\partial {#2}^{#1}}}

\title[Polynomial-Value Sieving and Recursively-Factorable Polynomials]{Polynomial-Value Sieving and \\ Recursively-Factorable Polynomials}

\author{Jonathan Burns}
\address{Department of Mathematics \& Statistics\\ University of South Florida\\ 4202 Fowler Ave.\\ Tampa, FL 33620-5700, USA}
\email{jtburns@mail.usf.edu}

\begin{document}

\maketitle


\begin{abstract}
We identify a recursive structure among factorizations of polynomial values into two integer factors. 
Polynomials for which this recursive structure characterizes all non-trivial representations of 
integer factorizations of the polynomial values into two parts are here called recursively-factorable polynomials. 
In particular, we prove that $n^2+1$ and the prime-producing polynomials $n^2+n+41$ and $2n^2+ 29$ are 
recursively-factorable. 

For quadratics, the we prove that this recursive structure is equivalent to a Diophantine identity involving the 
product of two binary quadratic forms. We show that this identity may be transformed into geometric terms, 
relating each integer factorization $a n^2+b n+c = p \, q$  to a lattice point of the conic section 
$a X^2 + b XY+cY^2+X-nY=0$, and vice versa.
\end{abstract}


\section{Introduction} \label{sec:intro}	

The sieve of Eratosthenes is the oldest and most well-known of the integer sieves, and is used to find all the primes
up to a given limit $N$. The sieve begins with the list of integers $L=(2, 3, \dots, N)$ and proceeds iteratively by marking 
the smallest number on the list as prime and removing it along with its multiples from the list. The smallest number
still left on the list is marked as prime and the procedure continues until the list is empty.

Algorithmically, the sieve of Eratosthenes both identifies the prime numbers in the list and yields a unique prime factorization
for the composite numbers through multiple presentations of each polynomial value as product of two integers. In other words,
each value $F(n)=n$ in the sequence $L=(F(2),F(3),\dots,F(N))$ is presented as the factorization presentation $F(n)=p \, q$ 
for each $p \mid F(n)$. If however $F$ is an arbitrary polynomial with integer coefficients and $p \mid F(n)$, then 
$p \mid F(n+ k \, p)$ for each $k \in \mathbb{Z}$ too. Hence, the algorithm can be generalized to include other polynomials 
at the cost of missing some of the factorization presentations. Fortunately, the situation can be improved by taking both factors 
of each composite $F(n)$ into consideration, i.e., if $F(n)= p \, q$ is marked as being divisible by $p$ then all $F(n + k q)$ where 
$k \in \mathbb{Z}$ can be marked as being divisible by $q$ as well. 

To keep track of all the factorization presentations, it suffices to record the initial value along with the sequence of quotients 
for the multiples of the factors, e.g., if $F(x_1) = 1 \cdot p_1$, $F(x_1 + x_2 p_1) = p_1 \, p_2$ and 
$F(x_1 + x_2 p_1 + x_3 p_2) = p_2 \, p_3$ then the factorization presentation can be reconstructed from the sequence 
$(x_1, x_2, x_3)$. This method of sieving the polynomial values for integer factorizations is expressed in Theorem \ref{main_recurrence},
and holds in the context of multivariate polynomials as well. Section \ref{sec:recursively-factorable} introduces a family of
polynomials called {\it recursively-factorable polynomials} for which the collection of factorization presentations corresponding 
to the sequences $\{(x_1,\dots,x_m) \in \mathbb{Z}^m\}_{m=1}^{\infty}$ yield the unique prime factorization for each 
value of $F$ via presentations $F(n)=p \, q$ for each $p \mid F(n)$. 

In general, recursively-factorable polynomials are rare, but there are some noteworthy instances. Particularly, the {\it Euler-like} 
and {\it Legendre-like} prime producing polynomials of the form $n^2+n+c$ for $c \in \{2,3,5,11,17,41\}$ and $2n^2 + c$ for 
$c = \{3,5,11,29\}$, respectively, and Landau's $n^2+1$ are recursively-factorable. The sieve of Eratosthenes verifies that the 
line $n$ is also recursively-factorable, but we presently focus on recursively-factorable quadratic equations. 

In Section \ref{sec:quad_dio}, we introduce an identity which presents the factorization of a quadratic polynomial value as the product
of two binary quadratic forms (Theorem \ref{thm:dio_main}) and show that this identity associates all the factorization presentations
of the aforementioned polynomial-value sieving integer sequences with the set 
$\Gamma_a := \left\{ \left. \begin{pmatrix}\alpha & \beta \\ \gamma & \delta \end{pmatrix} \in M_2(\mathbb{Z}) \, \right| \alpha \delta - a \, \beta \gamma = 1 \right\}$.
For monic quadratics, $a=1$ and the factorization presentations correspond to the transvection generators of 
$\Gamma_1 = \mbox{SL}_2(\mathbb{Z})$ (Corollary \ref{cor:transvections}).

In Section \ref{sec:lattice_conic}, a bijection is established (Theorem \ref{thm:conic_bijection}) between $\Gamma_a$ 
and the set $\mathcal{L}_a$ of lattice point solutions $(X,Y) \in \mathbb{Z}^2$ for the conic sections 
$a \, X^2 + b \, X Y + c \, Y^2 + X - n Y = 0$ with $a, b, c, n \in \mathbb{Z}$, showing that $\mathcal{L}_a$
does not depend on $b$, $c$, or $n$. Following the mappings in Figure \ref{fig:outline}, each lattice point $(X,Y)$
of the conic section is associated with an element of $\Gamma_a$ and gives a factorization presentation for 
$F(n)=a n^2+b n+c$. If a factorization presentation $F(n)=p \, q$ has a corresponding integer sequence $(x_1,\dots,x_m)$ then
there is a matching element of $\Gamma_a$ which corresponds to a lattice point solution of the conic section.

\begin{figure}[h]
\begin{tikzpicture}
\matrix (magic) [matrix of nodes,ampersand replacement=\&, column sep=3.5em]
{ 
$F(n)=p \, q$ \& \& \\
 \& $\begin{pmatrix} \alpha & \beta \\ \gamma & \delta \end{pmatrix} \in \Gamma_a$ \&  $(X,Y)\in \mathcal{L}_a$ \\ 
$(x_1,\dots,x_m)$ \& \& \\
};
\path[->]
(magic-2-2) edge [bend left=15] node[auto] {\scriptsize Thm. \ref{thm:conic_bijection}} (magic-2-3)
(magic-2-3) edge [bend left=15] node[auto] {\scriptsize Thm. \ref{thm:conic_bijection}} (magic-2-2)
(magic-2-2) edge [bend right=15] node[above] {\hspace{1.5em} \scriptsize Thm. \ref{thm:dio_main}} (magic-1-1)
(magic-3-1) edge [bend right=15] node[below] {\hspace{1.5em} \scriptsize Thm. \ref{thm:dio_quad}} (magic-2-2)
(magic-3-1) edge [bend left=15] node[auto]  {\scriptsize Thm. \ref{main_recurrence}} (magic-1-1);
\path[->,loosely dashed] (magic-1-1) edge [bend left=15] node[auto] {\scriptsize Thm. \ref{thm:all_factors}} (magic-3-1);
\end{tikzpicture}
\caption{Relationships between factorization presentations $a n^2+bn+c=p\, q$, the polynomial-value sieving sequence $(x_1,\dots,x_m)$, 
the set of $2 \times 2$ integers matrices $\Gamma_a$, and the set of lattice point solutions $\mathcal{L}_a$ to the conic section $a X^2+b X Y+cY^2+X-nY=0$.}
\label{fig:outline}
\end{figure}
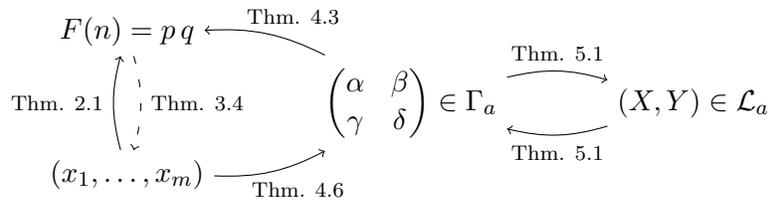

\section{Polynomial-Value Sieving} \label{sec:poly-value}	

\begin{theorem} \label{main_recurrence}
Let $\mathcal{R}$ be a commutative ring with identity. For any polynomial $F \in \mathcal{R}[x]$ of degree $d$, there exists a sequence of 
multivariate polynomials $\{f_m (x_1,\dots,x_m)\}_{m=0}^{\infty}$ such that $f_m (x_1, \dots, x_m) \in \mathcal{R}[x_1,\dots,x_m]$ and
\begin{equation} \label{gen_factor}
F \left( \sum\limits_{k=1}^m x_k \, f_{k-1}(x_1,\dots,x_{k-1}) \right) = f_{m-1}(x_1,\dots,x_{m-1}) \, f_m (x_1,\dots,x_m)
\end{equation}
where $f_0 = 1$, $f_1(x_1)=F(x_1)$, and 
$$f_m = f_{m-2} + x_m \sum\limits_{j=1}^d \frac{1}{j!} \left(x_m \, \frac{f_{m-1}}{f_{m-2}}\right)^{j-1} \frac{\partial^j f_{m-1}}{\partial x_{m-1}^j}$$
for $m \ge 2$ with the convention that $f_m$ is shorthand for $f_m(x_1, \dots, x_m)$.
\end{theorem}

\begin{proof}
Since $F(x_1 f_0) = f_0 \, f_1(x_1)$ represents the trivial factorization, the statement is initially true and we proceed by induction on $m$.
Let $D^{(j)}$ be the $j$th order Hasse derivative and $D^{(j)}_{x} =\frac{1}{j!} \frac{d^j}{dx^j}$ be the $j$th order Hasse derivative 
with respect to the intermediate $x$. Applying $D^{(j)}_{x_{m-1}}$ to both sides of
$F \left( \sum \limits_{k=1}^{m-1} x_k f_{k-1} \right)=f_{m-2} \, f_{m-1}$ gives
\begin{equation} \label{partial_f}
(D^{(j)} F) \left(\sum\limits_{k=1}^{m-1} x_k f_{k-1} \right) \cdot f_{m-2}^j = f_{m-2} \cdot D^{(j)}_{x_{m-1}} f_{m-1} .
\end{equation}
Using the Taylor series expansion for $F$,
\begin{align}
F\left(\sum\limits_{k=1}^{m} x_k f_{k-1} \right) & = \sum\limits_{j=0}^{d} (D^{(j)} F) \left(\sum\limits_{k=1}^{m-1} x_k f_{k-1}\right) \cdot (x_m f_{m-1})^{j} \notag \\ 
& = F\left( \sum\limits_{k=1}^{m-1} x_k f_{k-1} \right) + (x_m f_{m-1}) \sum\limits_{j=1}^{d} (D^{(j)} F) \left(\sum\limits_{k=1}^{m-1} x_k f_{k-1}\right) \cdot (x_m f_{m-1})^{j-1} \notag \\
& = f_{m-1} \cdot \left(f_{m-2} + x_m \sum\limits_{j=1}^{d} (D^{(j)} F) \left(\sum\limits_{k=1}^{m-1} x_k f_{k-1}\right) \cdot (x_m f_{m-1})^{j-1} \right) \label{hasse_taylor}
\end{align}
which gives a definition for $f_m (x_1,\dots,x_m) \in \mathcal{R} [x_1,\dots, x_m]$. Substituting (\ref{partial_f}) into (\ref{hasse_taylor}) yields
\begin{equation} \label{f_m}
f_m = f_{m-2} + x_m \sum\limits_{j=1}^{d} \left(x_m \frac{f_{m-1}}{f_{m-2}}\right)^{j-1} D^{(j)}_{x_m-1} f_{m-1} \, . \qedhere
\end{equation}
\end{proof}

\medskip

\begin{remark} \label{quad_recurrence}
For $F(z)=\sum\limits_{i=0}^d a_i z^i$, taking $j=d$ in equation (\ref{partial_f}) gives
$$\frac{D^{(d)}_{x_{m-1}} f_{m-1}}{(f_{m-2})^{d-1}} = a_d$$
for all $d \ge 1$. So for $d=2$,  Theorem \ref{main_recurrence} expresses $f_m$ as
\begin{equation}
f_m = f_{m-2} + x_m \frac{\partial f_{m-1}}{\partial x_{m-1}} + a_2 \, x_m^2 f_{m-1} \, .
\end{equation}
\end{remark}

\begin{remark}
For each sequence $(x_1, \dots, x_m)$, if $x_i = x_{i_a} + x_{i_a}$ then
\begin{equation} \label{binary_sequence}
f_m (x_1, x_2, \dots, x_{i-1}, x_i, x_{i+1}, \dots, x_m) = f_{m+2} (x_1, x_2, \dots, x_{i-1}, x_{i_a}, 0, x_{i_b}, x_{i+1} \dots, x_m) \, .
\end{equation}
Moreover if $(x_1, \dots, x_m) \in \mathbb{Z}^m$, then there exists $f_{M}$ such that
$$f_{m} (x_1, \dots, x_{m}) = f_{M} (z_1, \dots, z_{M})$$
where $z_i \in \{ -1,0,1 \}$ and $M = \sum_{j=1}^{m} 2|x_j|-1$.
\end{remark}

\begin{example} 
\label{recur_ex}
Let $F(x)=3x^2+5x+11$. We compute $f_3 (2,-1,4)$ as follows:

\begin{align*}
f_0 & = 1 \\
f_1 (2) & = \frac{F(2 \cdot 1) }{1} = \frac{F(2)}{1} = 33 \\
f_2 (2,-1) & = \frac{F(2 + (- 1) \cdot 33 )) }{33}=\frac{F(-31)}{33} =83 \\
f_3 (2,-1,4) & = \frac{F(-31 + 4 \cdot 83) }{83}=\frac{F(301)}{83} =3293\\
\end{align*}
This gives $F(301) =  273319 = 83 \times 3293$. One can also verify that $$f_3 (2,-1,4) = f_{11} (1,0,1,-1,1,0,1,0,1,0,1).$$
\end{example}

\smallskip

\section{Recursively-Factorable Polynomials}  \label{sec:recursively-factorable} 

Theorem \ref{main_recurrence} provides a means of factoring the values of a polynomial $F$ into two integers, but these presentations 
may not represent the full solution set $\{(n,p,q) \in \mathbb{Z}^3 : F(n) = p \, q \}$. For example when $F(n)=n^2+n+7$, the 
integer factorization $F(1)= 3 \cdot 3$ cannot be presented via Theorem \ref{main_recurrence}, i.e.,  there does not exist a finite 
sequence of integers $(x_1, x_2, \dots, x_m)$ for which $f_m = 3$, $f_{m-1} = 3$, and $\sum_{k=1}^m x_k f_{k-1} = 1$. Proof 
of this fact is shown in Remark \ref{non-recur-fac}.

By contrast, Lemma \ref{lem:recur_exist} provides the existence of a family of polynomials $\mathcal{F}$ for which the prime integer 
factorization of each value of $F \in \mathcal{F}$ can be reconstructed from the presentations of Theorem \ref{main_recurrence}. 
Theorem \ref{thm:all_factors} shows that this family of polynomials contains the recursively-factorable polynomials characterized 
by the following property.

\begin{definition}
Let $F$ be a polynomial with integer coefficients. If for each integer factorization presentation $F(n)=p \, q$ there exists an 
$r \in \mathbb{Z}$ such that $|F(r)|<|F(n)|$ and $r \equiv n \pmod{|p|}$ or $r \equiv n \pmod{|q|}$, then $n$ is said to satisfy
the {\it recursively-factorable criterion} for $F$. If each $n \in \mathbb{Z}$ satisfies the recursively-factorable criterion for $F$,
then the polynomial $F$ is said to be {\it recursively-factorable}.
\end{definition}

\begin{remark}
Recursively-factorable polynomials are irreducible over $\mathbb{Z}$. If not then $F(n)=0$ for some $n \in \mathbb{Z}$, but 
the non-trivial factorization $0 = 0 \cdot p_0$ has no associated $r \equiv n \pmod{|p_0|}$ such that $|F(r)|<|F(n)|=0$ for any 
$p_0 \in \mathbb{Z}$.
\end{remark}

\begin{lemma} \label{lem:shift_still_recursive}
Let $F$ be a polynomial and $G(n)= \pm F(n-h)$ for some $h \in \mathbb{Z}$. If $F$ is recursively-factorable, then so is $G$.
\end{lemma}

\begin{proof}
Suppose that $G(n)=\pm F(n-h)=p_0 \, p_1$ is a non-trivial factorization. Since $F$ is recursively-factorable, we may assume without loss 
of generality that there exists $q \in \mathbb{Z}$ such that $|F(r)|<|F(n-h)|$ where $r = (n - h) - q \, p_0$. Thus $|G(r+h)|<|G(n)|$
and $r+h = n- q \, p_0 \equiv n \pmod{|p_0|}$, so we may conclude that $G$ is recursively-factorable.
\end{proof}

\begin{theorem}
\label{thm:all_factors}
If $F$ is recursively-factorable then, for each $n\in \mathbb{Z}$ and $p \in \mathbb{N}$ such that $p \mid F(n)$, there exists a finite sequence 
of integers  $(x_1, x_2, \dots, x_m)$ such that 
\begin{equation}
n= \sum_{k=1}^m x_k f_{k-1} (x_1,\dots,x_{k-1}) \qquad \mbox{ and } \qquad p= |f_{m} (x_1,\dots, x_{m-1}, x_{m})|.
\end{equation}
\end{theorem}

\begin{proof}
Fix $n \in \mathbb{Z}$. If $p=1 \mbox{ or } |F(n)|$ then the sequence $(n)$ gives the presentation 
$F(n) = F(n \cdot f_0) = f_0 \, f_1 (n) = 1 \cdot F(n)$. Thus it is sufficient to consider the case where $F(n)$ 
is a composite integer with a non-trivial factorization $F(n) = p_1 \, p_0$ such that $p=|p_0|$. 

Let $R=\{r \in \mathbb{Z} : r \equiv n \pmod{|p_0|} \mbox{ or } r \equiv n \pmod{|p_1|} \}$. Since $F$ is recursively-factorable, there exists an
$r \in R$ such that $|F(r)|<|F(n)|$. Moreover there is an $r_1 \in R$ such that $|F(r_1)| \le |F(r)|$ for all $r \in R$. Set $p_* = p_0 \mbox{ or } p_1$
so that $r_1 \equiv n \pmod{|p_*|}$. It follows that $n=q_1 \, p_* + r_1$  and $F(r_1)=p_2 \, p_*$ for some $q_1, p_2 \in \mathbb{Z}$. 
If $|p_2| = 1$, then $F(r_1)=p_2 \, p_*$ represents a trivial factorization and the sequence $(r_1,q_1)$ yields the presentation
\begin{equation}
F(n) = F(r_1 \, p_2 + q_1 \, p_*) = f_1 (r_1) \, f_2 (r_1,q_1).
\end{equation}
If $|p_2| \not= 1$, then $F(r_1) = p_* \, p_2$ represents a non-trivial factorization, and by the minimality of our choice of $r_1$ 
relative to all other $r \in R$ there exists an $r_2$ which minimizes $|F(r_2)|<|F(r_1)|$ over all $r_2 \equiv r_1 \pmod{|p_2|}$, i.e.,
$r_1 = q_2 \, p_2 + r_2$ for some $q_2 \in \mathbb{Z}$.

We may continue in this fashion until we obtain the trivial integer factorization $F(r_{m-1}) = p_{m-1} \, p_{m}$ where $|p_m| = 1$, 
produced from a finite sequence of factors $(p^*, p_2, \dots, p_{m-1}, p_{m})$, quotients $(q_1, q_2, \dots, q_{m-1})$ and remainders 
$(r_1, r_2, \dots, r_{m-1})$ such that $r_{k} = q_{k+1} \, p_{k+1} + r_{k+1}$ and $F(r_k) = p_{k} \, p_{k+1}$ for each 
$2 \le k \le m-1$. Starting with $p_{m} = 1$ and $F(r_{m-1}) = p_{m-1} \, p_{m}$ we may reverse this sequence to obtain 
$n$ and $p$ as follows:
\begin{align*}
p_{m-1}& = \frac{F(r_{m-1})}{p_{m}} = \frac{F(r_{m-1})}{f_{0}} = f_1 (r_{m-1}), \\[10pt]
p_{m-2} &= \frac{F(r_{m-2})}{p_{m-1}} = \frac{F(r_{m-1} + q_{m-1} \, p_{m-1})}{p_{m-1}} =  \frac{F(r_{m-1} \, f_0 + q_{m-1} \, f_1 (r_{m-1}))}{f_1 (r_{m-1})} = f_2 (r_{m-1}, q_{m-1}).
\end{align*}
More generally $$p_k = f_{m-k} (r_{m-1}, q_{m-1}, q_{m-2}, \dots,q_{k+1})$$
for $2 \le k \le m-2$ and $p_*=f_{m-1} (r_{m-1},q_{m-1},\dots,q_2)$. 

Therefore the integer sequence $(r_{m-1},q_{m-1},\dots,q_1)$ gives the presentation
\begin{align*}
F(n) & = F \left( r_{m-1} \, f_0 + q_{m-1} f_1 (r_{m-1}) + \sum\limits_{k=3}^{m} q_{m-k+1} f_{k-1} (r_{m-1},q_{m-1}, \dots, q_{m-k+2}) \right) \\
& = f_{m-1} (r_{m-1},q_{m-1}, \dots, q_{2}) \, f_{m} (r_{m-1},q_{m-1}, \dots, q_2, q_{1}),
\end{align*}
and $p=|f_m (r_{m-1},q_{m-1},\dots,q_1)|$.
\end{proof}

\begin{figure}[h]
  \begin{center}
    \begin{overpic}[unit=1mm,width=4in]{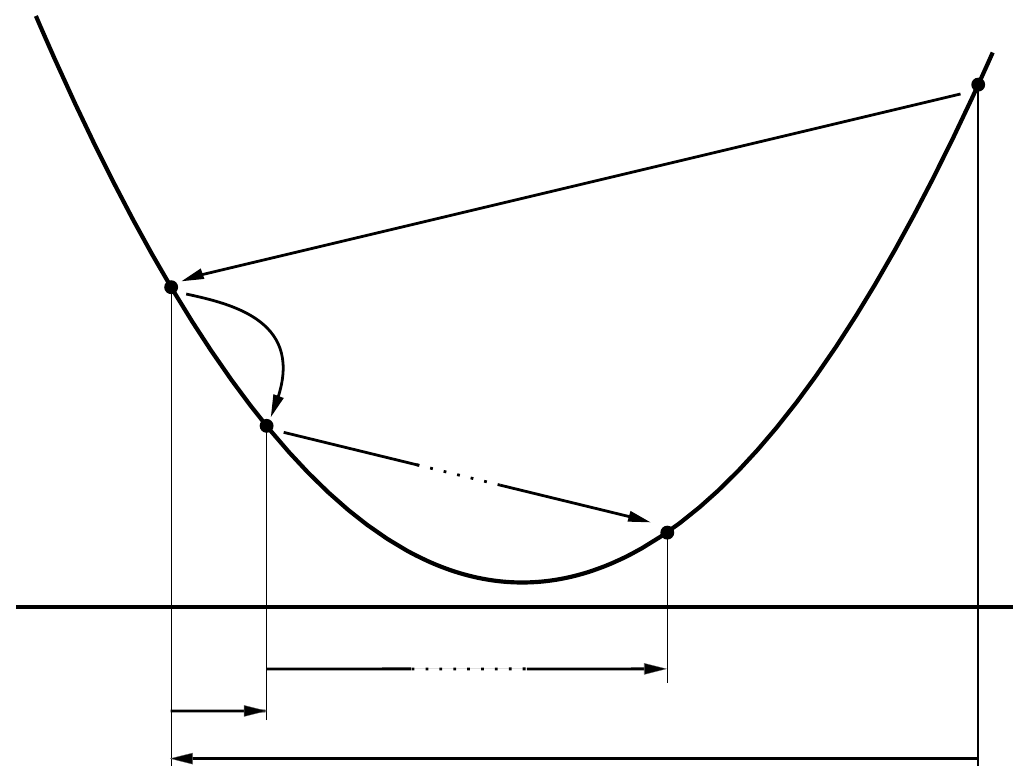}
      \put(99,67){$F(n)=p_1 \, p_0$}
      \put(-8,48){$F(r_1)=p_2 \, p_1$}
      \put(29,37.5){$F(r_2)=p_3 \, p_2$}
      \put(68,22){$F(r_k)=p_{k+1} \, p_k$}
      \put(55,13){\smaller $q_k \, p_k$}
      \put(18,10){\smaller $q_2 \, p_2$}
      \put(70,4){\smaller $q_1 \, p_1$}
    \end{overpic}
  \end{center}
\medskip
\caption{Sequence of decreasing values of $F$ used to compute $x_1, x_2, \dots, x_m$ in $f_m (x_1, x_2, \dots, x_m)$.}
\label{smaller_values}
\end{figure}
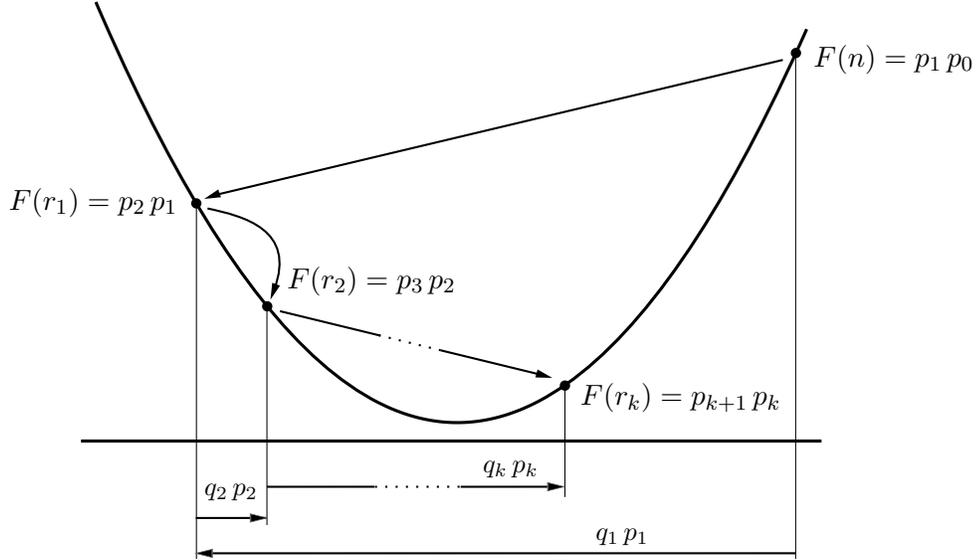

The proof of Theorem \ref{thm:all_factors} starts with an integer factorization $F(n_0)=p_1 \, p_0$ and 
constructs a sequence of factorizations $F(n_1) = p_1 \, p_2$, $F(n_2) = p_2 \, p_3, \dots$ such that 
$|F(n_0)| > |F(n_1)| > |F(n_2)| \dots$ until a prime number $F(n_m)$ with the trivial factorization $F(n_m) \cdot 1$ 
is reached. In this way prime-producing polynomials, which contain a large interval of consecutive prime values, 
make good candidates for having the recursively-factorable property.

In 1772, Euler \cite{Euler2} discovered that the polynomial $n^2-n+41$ produces
prime numbers for $n \in  [-39,40]$, and later Legendre \cite{Legendre}
noted that both $n^2+n+17$ and $n^2+n+41$ are prime for $n \in [-16,15]$ and 
$n=[-40,39]$, respectively. Le Lionnais considered polynomials of the type 
$n^2+n+\varepsilon$ in general, which he called Euler-like polynomials \cite{LeLionnais}, 
and integers $\varepsilon$ for which $n^2+n+\varepsilon$ is prime for 
$n=0,1,\dots,\varepsilon-2$ have come to be known as {\it lucky numbers of Euler}.

Rabinowitz \cite{Rabinowitz} proved that $\varepsilon$ is a lucky number of Euler if and only if
the field $\mathbb{Q}(\sqrt{1-4 \varepsilon})$ has class number 1. From this, Heegner \cite{Heegner} 
and Stark \cite{Stark1} showed that there are exactly six lucky numbers of Euler, namely 
2, 3, 5, 11, 17, and 41.

Legendre \cite{Legendre} explored other types of prime-producing quadratics such as $2n^2+\lambda$
which is prime when $\lambda = 29$ for $n = 0, 1, \dots ,28$. Akin to the Euler-like polynomials, these
quadratics give primes for $n = 0, 1, \dots, \lambda - 1$ for prime $\lambda$ if and only if 
$\mathbb{Q}(\sqrt{-2 \lambda})$ has class number 2 \cite{Frobenius, Louboutin}. Baker \cite{Baker} and 
Stark \cite{Stark2} found that the only such $\lambda$ are 3, 5, 11, and 29.

As seen in Lemma \ref{lem:recur_exist}, Euler-like and Legendre-like prime-producing quadratics are indeed
recursively-factorable. Further discussion of prime-producing quadratics can be found in \cite{Mollin,Ribenboim}.

\begin{lemma} \label{lem:recur_exist} \renewcommand{\labelenumi}{(\roman{enumi})}
The following quadratics (and their horizontal shifts) are recursively-factorable:
\begin{enumerate}
\item $n^2 + c$ \, where \,  $c \in \{1,2\}$,
\item $n^2 + n + c$ \,  where  \, $c \in \{1, 2, 3, 5, 11, 17, 41\}$ \label{Euler-like}
\item $2n^2 + c$ \, where \, $c \in \{1, 3, 5, 11, 29 \}$, \label{Legendre-like}
\item $2n^2 + 2n + c$ \, where \, $c \in \{1, 2, 3, 7, 19 \}$,  \label{2n^2+2n-like}
\item $3n^2 + c$ \, where \, $c =2$, 
\item $3n^2 + 3n + c$ \, where \, $c \in \{1, 2, 5, 11, 23\}$,
\item $4n^2 + c$ \, where \, $c \in \{1, 3, 7\}$, and
\item $4n^2 + 4n + c$ \, where \, $c \in \{2, 3, 5\}$.
\end{enumerate}
\end{lemma}

\begin{proof} 
We claim that if $F$ is one of these polynomials and all the values within a suitably large interval $I_{\hat{n}}$ 
are known to satisfy the recursively-factorable criterion for $F$, then the remaining values outside 
of $I_{\hat{n}}$ also satisfy the recursively-factorable criterion.

Supposing $F(n)=a n^2+b n+c$ is one of the polynomials in cases (i)-(viii), $F$ is a positive parabola having a 
minimum at either $n=0$ or $n=-\frac{1}{2}$. Furthermore the values $F(n) = F\left(-n-\frac{b}{a}\right)$ for all 
$n \in \mathbb{Z}$, so if $n$ satisfies the recursively-factorable criterion then so does $-n-\frac{b}{a}$.
Also note that $|F(m)|<|F(n)|$ for $m \in I_n = \left(\min\{-n-\frac{b}{a}, \, n\},\max\{-n-\frac{b}{a}, \, n\}\right)$. 

For cases (i)-(vi), define $\hat{n}$ such that $|2 \, n+\frac{b}{a}| > \lfloor \sqrt{F(n)} \rfloor$ for each $n \ge \hat{n}$.
Given that for each factorization presentation $F(n) = p \, q$ either $p \le \lfloor \sqrt{F(n)} \rfloor$ 
or $q \le \lfloor\sqrt{F(n)} \rfloor$, for $n \ge \hat{n}$ there exists a  $k \in \mathbb{Z}$ such that 
either $n - k \, p \in I_{\hat{n}}$ or $n - k \, q \in I_{\hat{n}}$. Thus if we can verify that the values within $I_{\hat{n}}$
satisfy the recursively-factorable criterion, then so do the values greater than $\hat{n}$ (and symmetrically
the values less than $-\hat{n}-\frac{b}{a}$), i.e., $F$ is recursively-factorable. In cases (vii) and (viii) we use a sharper
approximation of $\min\{p,q\}$ than $\lfloor \sqrt{F(n)} \rfloor$ to determine $\hat{n}$, but the idea is the same.

In cases (i), (iii), (v), and (vii), $F(n)$ is prime (or 1) for $n \in [1-c,c-1]$ and $c \mid F(\pm c)$ 
which means $c \mid F(0)=c$, so the recursively-factorable condition is satisfied for $n \in [-c,c]$. Similarly, 
$F(n)$ is prime (or 1) for $n \in [1-c,c-2]$ in cases (ii), (iv), (vi), and (viii). The recursively-factorable condition is 
satisfied for $-c$, $c-1$, and $c$ since $c \mid F(-c),F(c-1),F(c)$ and $F(0)=F(-1)=c$. Hence for all cases 
(i)-(viii) the recursively-factorable criterion is satisfied for $n \in [-c,c]$.

\medskip

\noindent \underline{Case (i)}: For $F(n)=n^2+c$ with $c \in \{1,2\}$, $\hat{n}=\lfloor\sqrt{\frac{c}{3}} \rfloor = 0$ 
and $I_{\hat{n}} = [0] \subset [-c,c]$.

\medskip

\noindent \underline{Case (ii)}: For $n^2 + n + c$ with $c \in \{1,2,3,5,11,17,41\}$, 
$I_{\hat{n}} = \left[-\left\lfloor -\frac{1}{2} + \sqrt{\frac{4c-1}{12}} \right\rfloor -1, \left\lfloor -\frac{1}{2} + \sqrt{\frac{4c-1}{12}} \right\rfloor \right]$
and yields the respective $I_{\hat{n}}$ intervals corresponding to each $c$: $[-1, 0] \subseteq [-1,1]$, 
$[-1,0]\subseteq [-2,2]$, $[-1,0]\subseteq [-3,3]$, $[-1,0]\subseteq [-5,5]$, $[-2, 1]\subseteq [-11,11]$, 
$[-2, 1]\subseteq [-17,17]$, and $[-4, 3]\subseteq [-41,41]$.

\medskip

\noindent  \underline{Case (iii)}: For $F(n)=2n^2 + c$ with $c \in \{1,3,5,11,29\}$, 
$I_{\hat{n}}=[-\lfloor \sqrt{\frac{c}{2}} \rfloor,\lfloor \sqrt{\frac{c}{2}} \rfloor]$ which gives the respective intervals:
$[0] \subseteq [-1,1]$, $[-1,1] \subseteq [-3,3]$, $[-1,1] \subseteq [-5,5]$, $[-2,2] \subseteq [-11,11]$, 
and $[-3,3] \subseteq [-29,29]$.

\medskip

\noindent  \underline{Case (iv)}:  Let $F(n) = 2n^2 + 2 n + c$ with $c \in \{1, 2, 3, 7, 19 \}$, 
$I_{\hat{n}} = \left[ - \lfloor \frac{\sqrt{2c -1} - 1}{2} \rfloor -1 , \lfloor \frac{\sqrt{2c -1} - 1}{2} \rfloor \right]$
which gives the respective intervals: 
$[0] \subseteq [-1, 1]$, $[0] \subseteq [-2, 2]$, $[0] \subseteq [-3, 3]$, $[-1,1] \subseteq [-7, 7]$, 
and $[-2, 2] \subseteq [-19, 19]$.

\medskip

\noindent  \underline{Case (v)}:  Let $F(n) = 3n^2 + 2$, then $I_{\hat{n}}=[0] \subseteq [-2,2]$.

\medskip

\noindent  \underline{Case (vi)}:  Let $F(n)=3n^2 +3n + c$ with $c \in \{1, 2, 5, 11, 23\}$, 
$I_{\hat{n}} = \left[ - \lfloor \frac{\sqrt{4c -3} - 1}{2} \rfloor -1 , \lfloor \frac{\sqrt{4c -3} - 1}{2} \rfloor \right]$
which gives the respective intervals: $[-1, 0] \subseteq [-1, 1]$, $[-1, 0] \subseteq [-2, 2]$, $[-2, 1] \subseteq [-5, 5]$,
$[-3, 2] \subseteq [-11, 11]$, and $[-5, 4] \subseteq [-23, 23]$.

\medskip

\noindent  \underline{Case (vii)}:  Let $F$ be of the form $4n^2 + c$ with $c \in \{1, 3, 7\}$. We claim that if
$F(n)=p \, q$ where $p \le q$ is an integer factorization presentation, then $p<2n$. Observe that $p=2n+1$
implies that $q\ge 2n+1$ and
$$4n^2+4n+1 = (2n+1)^2 \le p \, q = F(n) =4n^2 + c \quad \Longrightarrow \quad 4n+1 \le c$$
and is a contradiction for $n > c$. Similarly, for $p=2n$ and $q \ge 2n+1$,
$$4n^2+2n = 2n \, (2n+1) \le p \, q = F(n) =4n^2 + c \quad \Longrightarrow \quad 2n \le c$$
and is also contradiction for $n > c$. Clearly $q \not= 2n$ since $4n^2 + c = F(n) \not= p \, q = (2n)^2 = 4 n^2$.
Thus we are guaranteed that $2n > p$ and there exists an $r \in (1-n,n-1)$ such that $r \equiv n \pmod{p}$.

\noindent  \underline{Case (viii)}:  Let $F$ be of the form $4n^2 + 4n + c$ with $c \in \{2, 3, 5\}$. As in case (vii),
we show that $p < 2n$ for each integer factorization presentation $F(n) = p \, q$ where $p \le q$. First notice that
taking $p=2n+2$ and $q \ge 2n+2$ leads to
$$4n^2 +8n+4 = (2n+2)^2 \le p \, q = F(n) = 4 n^2 + 4 n + c \quad \Longrightarrow \quad 4n+4 \le c$$
and is a contradiction for $n > c$. Likewise, taking $p=2n+1$ and $q \ge 2n+2$ gives
$$4n^2 +6n+2 = (2n+1)(2n+2) \le p \, q = F(n) = 4 n^2 + 4 n + c \quad \Longrightarrow \quad 2n+2 \le c$$
and again is a contradiction for $n > c$. With $p =2n+1$ and $q=2n+1$, $4n^2 + 4n +c = F(n) \not= p \, q = 4n^2+4n+1$
as $c \not=1$. Finally assume that $p=2n$ and $q \ge 2n+3$,
$$4n^2 + 6n = (2n) (2n+3) \le p \, q = F(n) = 4 n^2 + 4 n + c \quad \Longrightarrow \quad 2n \le c$$
and is a contradiction for $n > c$. Finally take $q=2n+2$ to get the contradiction 
$4n^2+4n = (2n)(2n+2) = p \, q \not= F(n) = 4n^2 +4n +c$. Therefore if the recursively factorable 
criterion holds for the values in the interval $[-c,c]$, then $2n > p$ and the criterion holds for the values
outside of the interval also.
\end{proof}

\begin{table}
\begin{tabular}{|l|l|}
\hline
		& $c \le 5000$ \\
\hline \hline
$n^2-c$	& 2, 3, 6, 7, 11, 14, 23, 38, 47, 62, 83, 167, 227, 398 \\
\hline
$n^2+n-c$	& 1, 3, 4, 5, 7, 8, 9, 10, 13, 14, 15, 17, 18, 19, 22, \\
		& 23, 25, 27, 28, 33, 37, 39, 43, 45, 49, 53, 59, 67, \\ 
		& 69, 73, 75, 79, 85, 87, 93, 103, 109, 113, 115, \\ 
		& 127, 129, 139, 153, 163, 169, 179, 193, 199, 205, \\
		& 213, 235, 269, 283, 313, 337, 349, 373, 385, 409, \\
		& 469, 499, 619, 643, 655, 763, 829, 865, 883, 997, \\
		& 1063, 1555 \\
\hline
$2n^2-c$	& 1, 3, 5, 7, 11, 13, 15, 19, 21, 29, 31, 35, 37, 47, \\
		& 55, 61, 67, 69, 79, 91, 101, 103, 133, 139, 157, \\
		& 159, 181, 199, 229, 283, 439, 571, 643, 661, 1069 \\
\hline
$2n^2+2n-c$ 	& 1, 2, 3, 5, 6, 7, 9, 10, 11, 14, 15, 17, 21, 23, 26, \\
			& 27, 29, 35, 38, 41, 43, 53, 63, 65, 71, 81, 83, 86, \\
			& 107, 113, 146, 149, 173, 185, 191, 215, 218, 223, \\
			& 251, 317, 323, 371, 413, 491, 743, 833 \\
\hline
$3n^2-c$		& 1, 2, 5, 10, 14, 29, 46, 106, 149 \\
\hline
$3n^2+3n-c$ 	& 1, 2, 3, 4, 5, 7, 8, 11, 13, 17, 19, 23, 29, 31, 37, \\
			& 41, 47, 55, 59, 65, 67, 79, 89, 95, 97, 107, 119, \\
			& 131, 157, 163, 173, 199, 229, 257, 275, 317, 325 \\
			& 457, 479, 635, 637, 1379\\
\hline
$4n^2-c$		& 1, 2, 3, 5, 7, 11, 13, 17, 19, 23, 33, 41, 47, 59, 83 \\
			& 107, 167, 227, 563 \\
\hline
$4n^2+4n-c$	& 1, 2, 3, 5, 6, 7, 10, 11, 13, 19, 21, 22, 27, 31, 37, \\
			& 43, 46, 51, 61, 67, 82, 85, 115, 127, 163, 166, 226, \\
			& 277, 397 \\
\hline
\end{tabular}
\caption{Recursively-factorable polynomials with real roots.}
\label{tab:real_roots_recur_fac}
\end{table}

\begin{remark}
With some additional casework to show that the values over a suitably large interval satisfy the recursively-factorable
criterion, it can also be shown that the polynomials in Table \ref{tab:real_roots_recur_fac} are recursively-factorable.
Some of these quadratics are prime-producing polynomials, or a horizontal shift of one, listed in \cite{Mollin} and \cite{Weisstein}.

For these real-root quadratics, the condition $|F(m)| < |F(n)|$ for $m \in [2-n,n-1]$ no longer holds as it did in Lemma \ref{lem:recur_exist}. 
However for $n > \max \left\{ \frac{-b-\sqrt{b^2+8a c}}{2a}, \frac{-b+\sqrt{b^2+8a c}}{2a} \right\}$,
$|F(m)|<|F(n)|$ for all $m \in \left( -n-\frac{b}{a}, n \right)$. Hence $\hat{n}$ can be chosen to be sufficiently large so that,
for all $n>\hat{n}$, both $|F(m)|<|F(n)|$ for $m \in \left(-n-\frac{b}{a}, n \right)$ and $\lfloor \sqrt{|F(n)|} \rfloor < |2n+\frac{b}{a}|$.
\end{remark}

\section{Presentation as the Product of Binary Quadratic Forms} \label{sec:quad_dio}	

We show in this section that, for quadratic polynomials, the factorization presentations of Theorem \ref{main_recurrence}, defined
recursively as $F \left( \sum_{k=1}^m x_k f_{k-1} \right) = f_{m-1} f_m$, may be expressed in a closed form as the product of 
two binary quadratic forms. Theorem \ref{thm:dio_quad} establishes that, in this context, each factorization presentation sequence 
$(x_1,\dots, x_m)$ corresponds with a particular $A_m \in M_2(\mathbb{Z})$.

\begin{definition} {\rm Fix $F(n) = a \, n^2 + b \, n + c$. Let $\Delta_F$, $\eta_F$, $\phi_{F,0}$, and $\phi_{F,1}$ be functions
from $M_2 (\mathbb{Z}) \rightarrow \mathbb{Z}$ defined such that for 
$A=\begin{pmatrix} \alpha & \beta \\ \gamma & \delta \end{pmatrix}$,
\begin{equation}
\begin{aligned}
\Delta_F[A] & = \alpha \, \delta - a \, \beta \, \gamma, \\
\eta_F [A] & = \alpha \, \gamma + b \, \beta \, \gamma + c \, \beta \, \delta, \\
\phi_{F,0} [A] & = {\alpha}^2 + b \, \alpha \, \beta + a \, c \, {\beta}^2,\\
\phi_{F,1} [A] & = a {\gamma}^2 + b \, \gamma \, \delta + c \, {\delta}^2,
\end{aligned}
\end{equation} }
and for natural $m$,
\begin{equation}
\phi_{F,m} [A]  = \begin{cases}\phi_{F,0}[A] & \mbox{for even } m \\ \phi_{F,1}[A] & \mbox{for odd } m \end{cases} .
\end{equation}
We suppress the $F$ when it is clear by the context, favoring the notation $\Delta[A]$, 
$\eta[A]$, $\phi_0[A]$, $\phi_1[A]$, and $\phi_m[A]$.
\end{definition}

\begin{definition}
For $a \in \mathbb{Z}$, let
\begin{equation}
\Gamma_a:=\left\{ A \in M_2(\mathbb{Z}) : \Delta[A] = 1 \right\} .
\end{equation}
In general, the set $\Gamma_a$ is not closed under matrix multiplication and does not contain its inverses. However the case
when $a=1$ is particularly noteworthy as $\Gamma_1 = \mbox{SL}_2(\mathbb{Z})$ is the special linear group.
\end{definition}

\begin{theorem}
\label{thm:dio_main}
Let $F:\mathbb{Z} \rightarrow \mathbb{Z}$ such that $F(x) = a \, x^2 + b \, x + c$.
For $\alpha, \beta, \gamma, \delta \in \mathbb{Z}$,
$$F( \alpha \, \gamma + b \, \beta \, \gamma + c \, \beta \, \delta ) = (\alpha^2 + b \, \alpha \, \beta +
a \, c \, \beta^2)(a \, \gamma^2 + b \, \gamma \, \delta + c \, \delta^2) $$
if and only if 
$\alpha \, \delta - a \, \beta \, \gamma = 1$ or $-1-\frac{b \, ( \alpha \, \gamma + b \, \beta \, \gamma + c \, \beta \, \delta )}{c}$,
i.e., for $A \in M_2 (\mathbb{Z})$, 
\begin{equation}
F( \eta [A] ) = \phi_0 [A] \, \phi_1 [A]
\end{equation}
if and only if $\Delta [A] = 1$ or $-1-\frac{b}{c} \, \eta [A]$.
\end{theorem}

\begin{proof}
By expanding both sides, one can verify that:
\begin{align*}
\lefteqn{F( \alpha \, \gamma + b \, \beta \, \gamma + c \, \beta \, \delta )  - (\alpha^2 + b \, \alpha \, \beta + a \, c \, \beta^2) 
    \, (a \, \gamma^2 + b \, \gamma \, \delta + c \, \delta^2)}  \\
	& = ( 1 -  (\alpha \, \delta - a \, \beta \, \gamma ) ) \,
(c \, (\alpha \, \delta - a \, \beta \, \gamma ) + (c + b \, ( \alpha \, \gamma + b \, \beta \, \gamma + c \, \beta \, \delta ) )). \qedhere
\end{align*}
\end{proof}

\begin{remark}
The set of matrices $\mathcal{K}_1 \subset \Gamma_a$ given by
\begin{equation} \label{K1}
\mathcal{K}_1 = \left\{ \begin{pmatrix} 1 & 0 \\ s & 1 \end{pmatrix}, \begin{pmatrix} -1 & 0 \\ s & -1 \end{pmatrix} \right\},
\end{equation}
$\mathcal{K}_2 \subset \Gamma_1$ and $\mathcal{K}_3 \subset \Gamma_{-1}$ given by
\begin{equation} \label{K2_and_K3}
\mathcal{K}_2 = \left\{ \begin{pmatrix} s & 1 \\ -1 & 0 \end{pmatrix}, \begin{pmatrix} s & -1 \\ 1 & 0 \end{pmatrix} \right\} 
\quad \mbox{and} \quad
\mathcal{K}_3 = \left\{ \begin{pmatrix} s & 1 \\ 1 & 0 \end{pmatrix}, \begin{pmatrix} s & -1 \\ -1 & 0 \end{pmatrix} \right\},
\end{equation}
respectively, correspond to the trivial factorization in Theorem \ref{thm:dio_main} for each $s \in \mathbb{Z}$.
\end{remark}

The Fibonacci-Brahmagupta identity has a long history in mathematics beginning with its first appearance
in Diophantus' {\it Arithmetica} (III, 19) \cite{Diophantus} c.250 in the form of $(p^2 + q^2)(r^2+ s^2) = (p r + q s)^2 + (p s - q r)^2$.
Later in c.628, Brahmagupta generalized Diophantus' identity by showing that numbers of the form 
$p^2 + c \, q^2$ are closed under multiplication. Brahmagupta's identity was popularized in 1225 upon its reprinting in Fibonacci's 
{\it Liber Quadratorum} \cite{Fibonacci} where the first rigorous proof of the identity appeared. Finally in 1770, Euler \cite{Euler1} 
further generalized Brahmagupta's identity by providing the parametric solution
\begin{equation} \label{Euler_Identity}
(a d \, p^2 + c e \, q^2) (d e \, r^2 + a c \, s^2) = a e (d \, p r \pm c \, q s)^2 + c d (a \, p s \mp e \, q r)^2
\end{equation}
for the Diophantine equation $A x^2 + B y^2 = C$ with composite $C$. In Corollary \ref{Fib-Brahm_Identity} we show that the case $b=0$ 
in Theorem \ref{thm:dio_main} corresponds to the case $d=e=1$ in Euler's Identity (\ref{Euler_Identity}).

\begin{corollary} \label{Fib-Brahm_Identity}
$$ a \, (\alpha \, \gamma +  c \, \beta \, \delta)^2 + c \, (\alpha \, \delta - a \, \beta \, \gamma)^2 
= (\alpha^2 + a \, c \, \beta^2) \, (a \, \gamma^2 + c \, \delta^2) $$
\end{corollary}

\begin{proof}
When $b=0$, $F(x)= a \, x^2 + c$ and
\begin{align*}
a \, (\alpha \, \gamma +  c \, \beta \, \delta)^2 + c \cdot 1^2 & = F(\alpha \, \gamma + c \, \beta \, \delta)  \\
	& = (\alpha^2 + a \, c \, \beta^2) \, (a \, \gamma^2 + c \, \delta^2)
\end{align*}
where $\alpha \, \delta - a \, \beta \, \gamma = 1$. Hence
\begin{equation*} a \, (\alpha \, \gamma +  c \, \beta \, \delta)^2 + c \, (\alpha \, \delta - a \, \beta \, \gamma)^2 
= (\alpha^2 + a \, c \, \beta^2) \, (a \, \gamma^2 + c \, \delta^2) \, . \qedhere
\end{equation*}
\end{proof}

\begin{theorem}
\label{thm:dio_quad}
For $F(n) = a \, n^2 + b \, n + c$ and $m \geq 0$,
$$f_{m} (x_1, \dots ,x_m) = \phi_{m} [A_m]$$
where $A_m \in \Gamma_a$  defined recursively by
$$A_0 = \begin{pmatrix} 1 & 0 \\ 0 & 1 \end{pmatrix} 
\quad \quad \mbox{ and } \quad \quad A_{k+1} 
= \begin{pmatrix} \alpha_{k+1} & \beta_{k+1} \\ \gamma_{k+1} & \delta_{k+1} \end{pmatrix} 
= A_{k}+x_{k+1} B_{k}$$ 
for $1 \le k \le m-1$ such that
$$B_{k} = \begin{cases}
\begin{pmatrix} a \, \gamma_k & \delta_k \\ 0 & 0 \end{pmatrix}  &  \mbox{ for odd } k\\
\\
\begin{pmatrix} 0 & 0 \\ \alpha_k & a \, \beta_k \end{pmatrix}  &  \mbox{ for even } k\\
\end{cases}.$$
\end{theorem}

\begin{proof}
We shall proceed by induction on $m$. For each $1 \le k \le m$, define $A_{k} \in \Gamma_a$ and 
$B_k$ recursively as stated in the hypothesis. Initially we see that $f_0 = 1 = \phi_{0} [A_0]$ and 
$f_1 = F(x_1) = \phi_{1} [A_1]$ satisfies the hypothesis. Now assume $f_{2j}=\phi_0 [A_{2j}]$ and 
$f_{2j+1}=\phi_1[A_{2j+1}]$ for each $0 \le j \le \lceil \frac{m}{2} \rceil$. Suppose $m=2j$ for some $j \ge 1$.
Remark \ref{quad_recurrence} gives 
\begin{equation}
\label{quad_partial_2j}
f_{2j}=f_{2j-2} + x_{2j} \pderiv{x_{2j-1}}  \big[ f_{2j-1} \big] + a \, x_{2j}^2 f_{2j-1}.
\end{equation}
By the induction hypothesis
\begin{equation} \label{quad_fac_2j-2}
f_{2j-2} = \phi_0[A_{2j-2}] = \alpha_{2j-2}^2 + b \, \alpha_{2j-2} \, \beta_{2j-2} + a c \, \beta_{2j-2}^2
\end{equation}
and
\begin{equation} \label{quad_fac_2j-1}
f_{2j-1} = \phi_1[A_{2j-1}] = a \, \gamma_{2j-1}^2 + b \, \gamma_{2j-1} \, \delta_{2j-1} + c \, \delta_{2j-1}^2.
\end{equation}
The partial derivative $\pderiv{x_{2j-1}}  \big[ \phi_1[A_{2j-1}] \big]$ may be evaluated through the equation 
$A_{2j-1} = A_{2j-2} + x_{2j-1} B_{2j-2}$. In particular 
$$\pderiv{x_{2j-1}}  \big[\gamma_{2j-1}\big] = \alpha_{2j-2} \quad \mbox{and} \quad \pderiv{x_{2j-1}}  \big[\beta_{2j-1}\big] = a \, \beta_{2j-2}$$ 
which yields
\begin{equation} \label{quad_der_2j-1}
\begin{aligned}
\pderiv{x_{2j-1}}  \big[ f_{2j-1} \big] 
	& = \pderiv{x_{2j-1}}  \big[\phi_1 [A_{2j-1}] \big]  \\
	& = \pderiv{x_{2j-1}} \left[ a \, \gamma_{2j-1}^2+ b \, \gamma_{2j-1} \delta_{2j-1} + c \, \delta_{2j-1}^2 \right] \\ 
	& = 2 \, a \, \gamma_{2j-1} \, \alpha_{2j-2} + b  \left( a \, \gamma_{2j-1} \, \beta_{2j-2} + \delta_{2j-1} \, \alpha_{2j-2} \right) + 2 \, a c \, \delta_{2j-1} \, \beta_{2j-2}
\end{aligned}
\end{equation}
Substituting (\ref{quad_fac_2j-2}), (\ref{quad_fac_2j-1}), and (\ref{quad_der_2j-1}) into (\ref{quad_partial_2j}) gives
\begin{equation}  \label{fac_comp_2j}
\begin{aligned}
f_{2j} & = (\alpha_{2j-2}^2 + b \, \alpha_{2j-2} \, \beta_{2j-2} + a c \, \beta_{2j-2}^2) \\
	& \quad +x_{2j} (2 \, a \, \gamma_{2j-1} \, \alpha_{2j-2} + a  b \, \gamma_{2j-1} \, \beta_{2j-2} + b \, \delta_{2j-1} \, \alpha_{2j-2}  \\
	& \quad + 2 \, a c \, \delta_{2j-1} \, \beta_{2j-2}) + x_{2j}^2 \, (a) ( a \, \gamma_{2j-1}^2 + b \, \gamma_{2j-1} \, \delta_{2j-1} + c \, \delta_{2j-1}^2).
\end{aligned}
\end{equation}
As defined in the hypothesis,
\begin{equation}
\begin{aligned}
A_{2j} & = A_{2j-1} + x_{2j} B_{2j-1}  \\
	& = (A_{2j-2} + x_{2j-1} B_{2j-2}) + x_{2j} B_{2j-1} \\
	& = \begin{pmatrix} \alpha_{2j-2} & \beta_{2j-2} \\ \gamma_{2j-2} & \delta_{2j-2} \end{pmatrix} 
+ \begin{pmatrix} 0 & 0 \\ x_{2j-1} \alpha_{2j-2} & a \, x_{2j-1}\beta_{2j-2} \end{pmatrix} 
+ \begin{pmatrix} a \, x_{2j} \gamma_{2j-1} & x_{2j} \delta_{2j-1} \\ 0 & 0 \end{pmatrix}\\
	& = \begin{pmatrix} 
	\alpha_{2j-2} + a \, x_{2j} \, \gamma_{2j-1} 	& 	\beta_{2j-2} + x_{2j} \, \delta_{2j-1} \\
	\gamma_{2j-2} + x_{2j-1} \, \alpha_{2j-2}	& 	\delta_{2j-2} + a \, x_{2j-1} \, \beta_{2j-2}
\end{pmatrix} 	
\end{aligned}
\end{equation}
so
\begin{equation}  \label{phi_comp_2j}
\begin{aligned}
\phi_{2j}[A_{2j}] 	& = \phi_0 [A_{2j}]\\
	& = (\alpha_{2j-2} + a \, x_{2j} \, \gamma_{2j-1})^2  + ac \, (\beta_{2j-2} + x_{2j} \, \delta_{2j-1})^2 \\
	& \quad + b \, (\alpha_{2j-2} + a \, x_{2j} \, \gamma_{2j-1})(\beta_{2j-2} + x_{2j} \, \delta_{2j-1})	.
\end{aligned}
\end{equation}
Comparing (\ref{fac_comp_2j}) and (\ref{phi_comp_2j}) shows that $f_{2j} = \phi_{2j}[A_{2j}]$. 

Initially $\Delta[A_0] = 1$ and by the induction hypothesis $\Delta [A_k] = 1$ for $1 \le k\le m-1$, so we check that $A_m \in \Gamma_a$:
\begin{align*}
\Delta[A_m]  & = \Delta \left[\begin{pmatrix}\alpha_{m-1} + x_m \, a \, \gamma_{m-1} & \beta_{m-1} + x_m \, \delta_{m-1} \\ \gamma_{m-1} & \delta_{m-1} \end{pmatrix} \right] \\
& = (\alpha_{m-1}+x_m \, a \, \gamma_{m-1}) \, \delta_{m-1} - a \, (\beta_{m-1} + x_m \, \delta_{m-1}) \, \gamma_{m-1} \\
& = (\alpha_{m-1} \, \delta_{m-1} - a \, \beta_{m-1} \, \gamma_{m-1}) = \Delta[A_{m-1}] = 1 \, .
\end{align*}

Similarly when
$m=2j+1$, Remark \ref{quad_recurrence} says that
\begin{equation}
\label{quad_partial_2j+1}
f_{2j+1}=f_{2j-1} + x_{2j+1} \pderiv{x_{2j}}  \big[ f_{2j} \big] + a \, x_{2j+1}^2 f_{2j}.
\end{equation}
whose partial derivative  $\pderiv{x_{2j}} \big[ f_{2j} \big] = \pderiv{x_{2j}} \big[ \phi_{2j} [A_{2j}] \big]$ 
may be computed through  (\ref{phi_comp_2j}) as
\begin{equation}  \label{quad_der_2j}
\begin{aligned}
\pderiv{x_{2j}}  \big[ f_{2j} \big]  & = 2 \, a \, \alpha_{2j} \, \gamma_{2j-1} + b \, \alpha_{2j} \, \delta_{2j-1}  
	+ a \, b \, \gamma_{2j-1} \, \beta_{2j} + 2 \, a \, c \, \beta_{2j} \, \delta_{2j-1}
\end{aligned}
\end{equation}
since $\alpha_{2j} = \alpha_{2j-2}+a \, x_{2j} \, \gamma_{2j-1}$ and $\beta_{2j} = \beta_{2j-2}+x_{2j} \, \delta_{2j-1}$.
Putting (\ref{quad_fac_2j-1}), (\ref{quad_partial_2j+1}), and (\ref{quad_der_2j}) together with the fact that 
$f_{2j} = \phi_0[A_{2j}]$ gives
\begin{equation}  \label{fac_comp_2j+1}
\begin{aligned}
f_{2j+1} & = (a \, \gamma_{2j-1}^2 + b \, \gamma_{2j-1} \delta_{2j-1} + c \, \delta_{2j-1}^2) \\
	& \quad + x_{2j+1} (2 \, a \, \alpha_{2j} \, \gamma_{2j-1} + b \, \alpha_{2j} \, \delta_{2j-1} 
	+ a \, b \, \gamma_{2j-1} \, \beta_{2j} \\
	& \quad + 2 \, a \, c \, \beta_{2j} \, \delta_{2j-1}) + x_{2j+1}^2 \, (a) (\alpha_{2j}^2
	+ b \, \alpha_{2j} \, \beta_{2j} + a \, c \, \beta_{2j}^2)
\end{aligned}
\end{equation}
and may be compared with $\phi_{2j+1} [A_{2j+1}]$ which is computed thusly:
\begin{equation} \label{phi_comp_2j+1}
\begin{aligned}
\phi_{2j+1} \big[ A_{2j+1} \big]  & = \phi_1 \left[ 
\begin{pmatrix} 
	\alpha_{2j-1} + a \, x_{2j} \, \gamma_{2j-1} 	& 	\beta_{2j-1} + x_{2j} \, \delta_{2j-1} \\
	\gamma_{2j-1} + x_{2j+1} \, \alpha_{2j}	& 	\delta_{2j-1} + a \, x_{2j+1} \, \beta_{2j} 
\end{pmatrix} \right]  \\
	& = a \, (\gamma_{2j-1} + x_{2j+1} \, \alpha_{2j})^2 
	+ c \, (\delta_{2j-1} + a \, x_{2j+1} \, \beta_{2j})^2	\\
	& \quad + b\, (\gamma_{2j-1} + x_{2j+1} \, \alpha_{2j})(\delta_{2j-1} + a \, x_{2j+1} \, \beta_{2k}).
\end{aligned}
\end{equation}
Checking that (\ref{fac_comp_2j+1}) is equal to (\ref{phi_comp_2j+1}) shows $f_m = \phi_m[A_m]$.

We have that $\Delta[A_k] = 1$ for $1 \le k \le m-1$, so
\begin{align*}
\Delta[A_m]  & = \Delta \left[\begin{pmatrix}\alpha_{m-1} & \beta_{m-1} \\ \gamma_{m-1} + x_m \, \alpha_{m-1} & \delta_{m-1} + x_m \, a \, \beta_{m-1} \end{pmatrix} \right] \\
& = \alpha_{m-1} \, (\delta_{m-1}+x_m \, a \, \beta_{m-1}) - a \, \beta_{m-1} (\gamma_{m-1}+ x_m \, \alpha_{m-1}) \\
& = (\alpha_{m-1} \, \delta_{m-1} - a \, \beta_{m-1} \, \gamma_{m-1}) = \Delta[A_{m-1}] = 1 \, .
\end{align*}
which completes the proof.
\end{proof}

Combining Theorems \ref{thm:all_factors} and \ref{thm:dio_quad} implies that for a recursively-factorable polynomial $F$, 
each non-trivial factorization presentation $(n, p, q \in \mathbb{Z}: |F(n)| = p \, q)$ is represented by some 
$A_m \in \Gamma_a$ via the identity $F(\eta[A_m])=\phi_0 [A_m] \, \phi_1 [A_m]$ from Theorem \ref{thm:dio_main}.

\begin{example}
Returning to Example \ref{recur_ex}, for $F(n)=3n^2+5n+11$ we can compute $f_3 (2,-1,4)$ using 
Theorem \ref{thm:dio_quad}:
\begin{align*}
A_1 & = \begin{pmatrix} 1 & 0 \\ 2 & 1 \end{pmatrix} \\
A_2 & = \begin{pmatrix} 1 & 0 \\ 2 & 1 \end{pmatrix} + (-1) \begin{pmatrix} 3 \cdot 2 & 1 \\ 0 & 0  \end{pmatrix} 
= \begin{pmatrix} -5 & -1 \\ 2 & 1 \end{pmatrix} \\
A_3 & = \begin{pmatrix} -5 & -1 \\ 2 & 1 \end{pmatrix} + (4) \begin{pmatrix} 0 & 0 \\ -5 & 3 \cdot (-1) \end{pmatrix} 
= \begin{pmatrix} -5 & -1 \\ -18 & -11 \end{pmatrix} 
\end{align*}
and $$f_3 (2,-1,4) = \phi_1 [A_3] = 3 \, (-18)^2 + 5 \, (-18) (-11) + 11 \, (-11)^2 = 3293.$$
It is readily checked that $\Delta [A_3] = 1$ and meets the conditions of Theorem \ref{thm:dio_main}. 
Since $\eta [A_3] = 301$ and $\phi_2 [A_3] = 83$, it follows that $$F(301) = 3293 \times 83.$$
\end{example}

\begin{remark} \label{non-recur-fac}
The non-trivial factorization $F(1)= 3 \cdot 3$, but $F(0)=7$ is the only value less than $F(1)$ and $1 \not\equiv 0 \pmod{3}$. 
Likewise $F(1) = 3 \cdot 3$ cannot be represented by Theorem \ref{thm:dio_main}, since $3$ cannot be represented 
by the binary form $\phi_0 [A] = \alpha^2 + \alpha \, \beta + 7  \, \beta^2$, see \cite{Conway} for more details. 
\end{remark}

\begin{remark}
Recall that the special linear group may be generated by its transvections \cite{Hahn}. In particular, 
$\mbox{SL}_2(\mathbb{Z})=\langle T, U \rangle$ where $T=\begin{pmatrix} 1 & 1 \\ 0 & 1 \end{pmatrix} $ 
and $U=\begin{pmatrix} 1 & 0 \\ 1 & 1 \end{pmatrix}$. It follows that
$$T^{i}=\begin{pmatrix} 1 & i \\ 0 & 1 \end{pmatrix} \quad \mbox{ and } \quad U^{i}=\begin{pmatrix} 1 & 0 \\ i & 1 \end{pmatrix}$$
for all $i \in \mathbb{Z}$.
\end{remark}

\begin{corollary} \label{cor:transvections}
For $F(n) = n^2 + b \, n + c$, 
\begin{equation}
f_m (x_1, x_2, \dots, x_{2i-1}, x_{2i}, \dots, x_m) = \phi_m [W^{x_m} \dots T^{x_{2i}} U^{x_{2i-1}} \dots T^{x_2} U^{x_1}]
\end{equation}
where $W=\left\{ \begin{array}{ll}U, & \mbox{if } m \mbox{ is odd} \\ T, & \mbox{if } m \mbox{ is even.} \end{array} \right.$
\end{corollary}

\begin{proof}
From Theorem \ref{thm:dio_quad}, $f_m = \phi_m[A_m]$ where $A_0 = I$ and
\begin{equation}
A_{k} = 
  \begin{cases} 
    \begin{pmatrix} \alpha_{k-1} & \beta_{k-1} \\ \gamma_{k-1} + x_k \alpha_{k-1} & \delta_{k-1} + x_k \beta_{k-1} \end{pmatrix}
      = U^{x_k} A_{k-1} & \mbox{for odd } k \\
    \begin{pmatrix} \alpha_{k-1}+ x_{k} \gamma_{k-1} & \beta_{k-1}  + x_{k} \delta_{k-1} \\ \gamma_{k-1} & \delta_{k-1} \end{pmatrix} 
      = T^{x_k} A_{k-1} & \mbox{for even } k
  \end{cases}
\end{equation}
for each $1 \le k \le m$.
\end{proof}

It stands to reason that shifting a polynomial horizontally does not change the integer factorization of its values. In the case of quadratics,
the specific correspondence between a parabola and its shift is expressed by the following proposition.

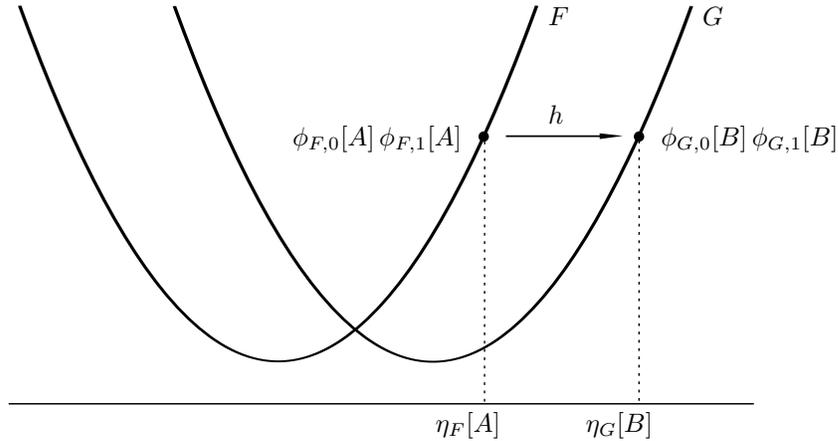
\begin{figure}[b]
  \begin{center}
    \begin{overpic}[unit=1mm,width=4in]{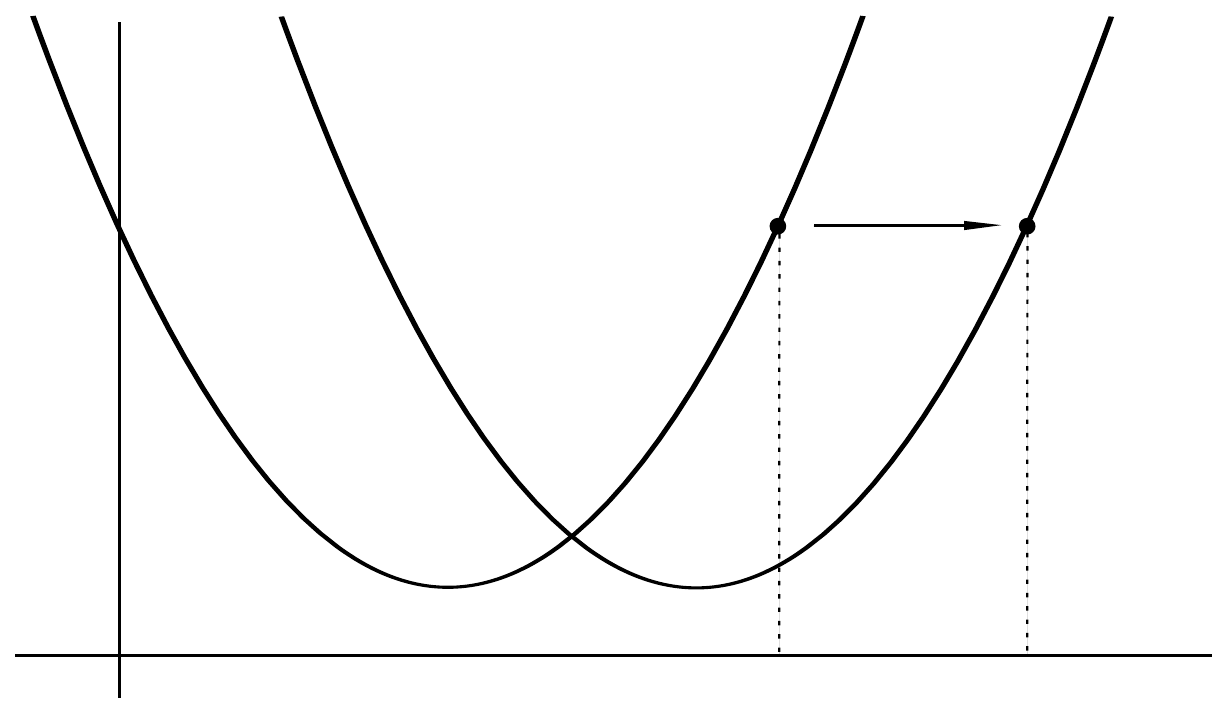}
      \put(73,55){\smaller $F$}
      \put(93.5,55){\smaller $G$}
      \put(73,42){\smaller $h$}
      \put(39,39){\smaller $\phi_{F,0}[A] \, \phi_{F,1}[A]$}
      \put(88,39){\smaller $\phi_{G,0}[B] \, \phi_{G,1}[B]$}
      \put(58,1){\smaller $\eta_F[A]$}
      \put(78,1){\smaller $\eta_G[B]$}
    \end{overpic}
  \end{center}
\medskip
\caption{Correspondence between integer factorizations for shifted parabolas.}
\label{smaller_values}
\end{figure}

\begin{proposition} \label{horz_shift}
Let $F(n) = a \, n^2 + b \, n + c$ and set $G(n) = F(n-h)$ for some $h \in \mathbb{Z}$. For each $A= \begin{pmatrix}
\alpha & \beta \\
\gamma & \delta
\end{pmatrix} \in \Gamma_a$ there is a corresponding $$B = A + h \,  \begin{pmatrix} 
a \, \beta & 0 \\
\delta & 0 \\
\end{pmatrix}$$
for which the following conditions hold:
\begin{enumerate} \renewcommand{\labelenumi}{(\roman{enumi})}
\item $B \in \Gamma_a$,
\item $\eta_G [B] = \eta_F [A] + h$,
\item $\phi_{G,0} [B] = \phi_{F,0} [A]$, and 
\item $\phi_{G,1} [B] = \phi_{F,1} [A]$.
\end{enumerate}
\end{proposition}

\begin{proof}
Let $B = \begin{pmatrix} \alpha + h \, a \beta & \beta \\ \gamma + h \, \delta & \delta \end{pmatrix}$ such that
$\alpha \delta - a \beta \gamma = 1$. Noting that 
$$G(n) = F(n-h) = a \, n^2 + (b - 2 a h) \, n + (c - b h + a h^2) :$$
\begin{align*}
\Delta_G [B] & = (\alpha + h \, a \beta) \, \delta - a \, \beta (\gamma + h \, \delta) \tag{i} \\
	& = \alpha \delta - a \, \beta \gamma = 1. \\[6pt]
\eta_G [B]  & = (\alpha + h \, a \beta) (\gamma + h \, \delta) + (b - 2 a h) \, \beta (\gamma + h \, \delta) + (c - b h + a h^2) \, \beta \delta \tag{ii} \\
	& = (\alpha \gamma + b \, \beta \gamma + c \, \beta \delta) + h \, (\alpha \delta - a \beta \gamma) \\
	& = \eta_F [A] + h . \\[6pt]
\phi_{G,1} [B] & = a \, (\gamma + h \, \delta)^2 + (b - 2 a h) (\gamma + h \, \delta) \delta + (c - b h + a h^2) \, \delta^2 \tag{iii} \\
	& = a \, \gamma^2 + b \, \gamma \delta + c \, \delta^2. \\[6pt]
\phi_{G,2} [B] & = (\alpha + h \, a \beta)^2 + (b - 2 a h) (\alpha + h \, a \beta) \beta + a (c - b h + a h^2) \, \beta^2 \tag{iv} \\
	& = \alpha^2 + b \, \alpha \beta + a c \, \beta^2 . \qedhere
\end{align*}

\end{proof}

\section{Lattice Points on the Conic Section $a X^2 + b X Y + c Y^2 + X - n Y =0$}  \label{sec:lattice_conic}	

Lastly, Theorem \ref{thm:conic_bijection} relates the set $\Gamma_a$ with the lattice point solutions of the conic sections
$a X^2 + b X Y+ c Y^2+ X -n Y = 0$. From Theorem \ref{thm:dio_main}, each $A_m \in \Gamma_a$ corresponds to an 
integer factorization presentation of a value of $F(n)=a n^2+b n +c$, i.e., the problem of finding lattice point solutions 
to these conic sections is equivalent to factoring the value of an associated quadratic polynomial.

\begin{theorem} \label{thm:conic_bijection}
For $a,b,c \in \mathbb{Z}$,  let
$$\mathcal{L}_a = \{(X,Y) \in \mathbb{Z}^2 \mid a X^2 + b X Y + c Y^2 + X - n Y = 0 \mbox{ for any } n \in \mathbb{N} \}$$
The map $\psi : \Gamma_a / \mathcal{K}_1 \bigcup \mathcal{K}_2 \bigcup \mathcal{K}_3 \rightarrow \mathcal{L}_a/\{(0,0),(-1,0),(1,0)\}$ defined by 
$$\begin{pmatrix} \alpha & \beta \\ \gamma & \delta \end{pmatrix} \mapsto 
\begin{pmatrix} \beta \gamma \\ \beta \delta \end{pmatrix}$$
is a bijection.
\end{theorem}

\begin{proof}
Fix $a,b,c \in \mathbb{Z}$ and consider $A= \begin{pmatrix} \alpha & \beta \\ \gamma & \delta \end{pmatrix} \in 
\Gamma_a$. Set $n = \eta [A]$, $X = \beta \gamma$, 
$Y = \beta \delta$, and $Z = \alpha \gamma$. Direct substitution shows that 
\begin{equation}
Z + b \, X + c \, Y = \alpha \gamma + b \, \beta \gamma + c \, \beta \delta = \eta [A]  = n.
\label{ellipse_1}
\end{equation}
Since $A \in \Gamma_a$, it follows that $\Delta[A]=1$ and $\beta \gamma \: (\alpha \delta-a \, \beta \gamma) = \beta \gamma (1)$, i.e.,
\begin{equation}
Z Y = X + a X^2.
\label{ellipse_2}
\end{equation}
Solving for $Z$ in (\ref{ellipse_1}) and substituting it into (\ref{ellipse_2}) shows that $(X,Y)$ is a solution to
\begin{equation}
a X^2 + b X Y + c Y^2 + X - n Y = 0.
\label{ellipse_3}
\end{equation}

Now consider the inverse map $\psi^{-1} : \mathcal{L}_a/\{(0,0),(-1,0),(1,0)\} \rightarrow \Gamma_a / \mathcal{K}_1 \bigcup \mathcal{K}_2 \bigcup \mathcal{K}_3$ defined by
\begin{equation}
\begin{pmatrix} X \\ Y \end{pmatrix} \mapsto \begin{pmatrix} 
\frac{\gcd(X,Y)}{Y}(1 + a X) & \gcd(X,Y) \\
\frac{X}{\gcd(X,Y)} & \frac{Y}{\gcd(X,Y)}
\end{pmatrix} \, .
\end{equation}
For each $L=(X,Y) \in \mathcal{L}_a$, $\Delta\left[ \psi^{-1}(L) \right] = 1$ and from (\ref{ellipse_3})
$$X (1 + a X) = Y (n - b X - c Y)$$
so $\frac{\gcd(X,Y)}{Y}(1+a X) \in \mathbb{Z}$. Hence $\psi^{-1} (L) \in \Gamma_a$.

We show that $\psi$ is injective by verifying that $\psi^{-1} \circ \psi (A) = A$ for each $A \in \Gamma_a$.
Indeed, since $\Delta[A]=1$ the $\gcd(\alpha \delta, a \, \beta \gamma)=1$ implying that $\gcd(\gamma,\delta)=1$, i.e., 
$\gcd(\beta \gamma, \beta \delta)=\beta$. Thus,
$$\psi^{-1} \psi [A] = \psi^{-1} \left[ \begin{pmatrix} \beta \gamma \\ \beta \delta  \end{pmatrix} \right]
   = \begin{pmatrix} 
\frac{\beta}{\beta \delta}(1 + a \beta \gamma) & \beta \\
\frac{\beta \gamma}{\beta} & \frac{\beta \delta}{\beta} \end{pmatrix} = A$$
since $\Delta[A]=1$ implies that $\alpha = \frac{1}{\delta}(1+ a \beta \gamma)$. 

Likewise, for each $(X,Y) \in \mathcal{L}_a$,
$$\psi \circ \psi^{-1} \left[ \begin{pmatrix} X \\ Y \end{pmatrix} \right]
= \psi \left[ \begin{pmatrix} 
\frac{G}{Y}(1 + a X) & G \\
\frac{X}{G} & \frac{Y}{G} \\
\end{pmatrix} \right] 
= \begin{pmatrix} X \\ Y \end{pmatrix}$$
meaning $\psi$ is surjective.
\end{proof}

The mapping $\psi:\mathcal{K}_1 \mapsto \begin{pmatrix} 0 \\ 0 \end{pmatrix}$ defined by 
$\psi \left[ \begin{pmatrix} \alpha & \beta \\ \gamma & \delta \end{pmatrix} \right] = \begin{pmatrix} \beta \gamma \\ \beta \delta \end{pmatrix}$ 
is well-defined and onto, but is not one-to-one.
Similarly, when $a=1$ or $-1$ the respective mappings $\psi:\mathcal{K}_2 \mapsto  \begin{pmatrix} -1 \\ 0 \end{pmatrix}$ and 
$\psi: \mathcal{K}_3 \mapsto \begin{pmatrix} 1 \\ 0 \end{pmatrix}$ are onto but not one-to-one. 
Therefore the image of $\psi$ under $\Gamma_a$ is $\mathcal{L}_a$.

\begin{figure}[h]
\begin{center}
  \includegraphics[width=5in]{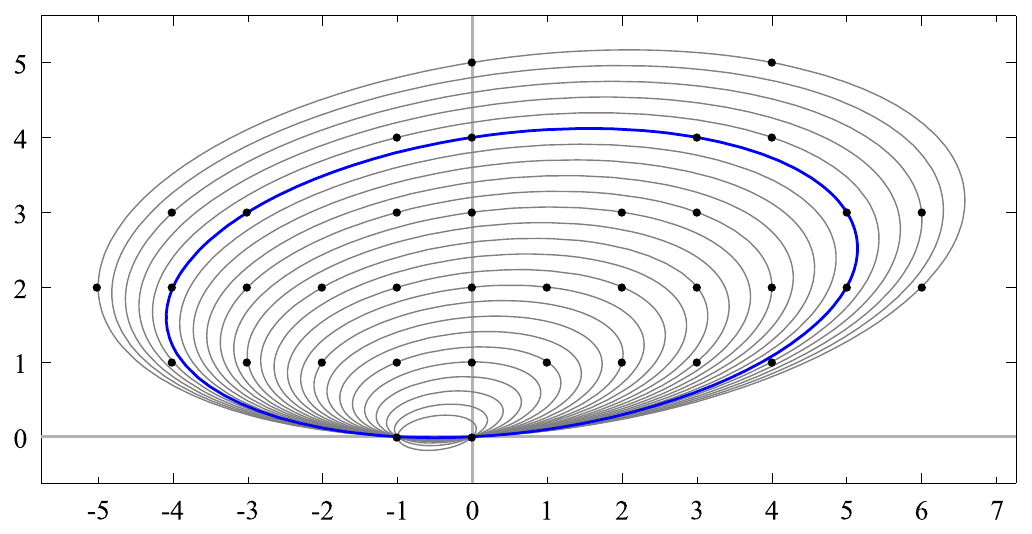}
\end{center}
\caption{Plot of $X^2 - X Y + 5 Y^2 + X - n Y = 0$ for $n=0,\dots,25$. The case $n=20$ is highlighted in blue and lattice points 
$(X,Y) \in \mathcal{L}_1$ intersecting the ellipses are indicated.}
\label{ellipse_plot}
\end{figure}

\begin{example}
Consider the Euler-like polynomial $F(n) = n^2 - n + 5$. It is easy to verify that $(X,Y)=(3,4)$ is a solution of
\begin{equation} \label{conic_example}
X^2 - X Y + 5 Y^2 + X - 20 Y = 0.
\end{equation}
By Theorem \ref{thm:conic_bijection}, the point $(3,4)$ corresponds to the element $A \in \Gamma_1$ given by
$$A = \psi^{-1} \left[ \begin{pmatrix} 3 \\ 4 \end{pmatrix} \right] =  \begin{pmatrix} 1 & 1\\ 3 & 4 \end{pmatrix}.$$
Thus $F(\eta[A]) = F(20) = 5 \cdot 77 = \phi_1 [A] \phi_2 [A]$. Similarly $(0,0)$, $(5,2)$, $(5,3)$, $(0,4)$, $(-3,3)$, $(-4,2)$ and $(-1,0)$ 
are also lattice point solutions (see Figure \ref{ellipse_plot}) to (\ref{conic_example}) corresponding to the integer factorizations 
$1 \cdot 385$, $11 \cdot 35$, $7 \cdot 55$, $77 \cdot 5$, $55 \cdot 7$, $35 \cdot 11$, and $385 \cdot 1$, respectively.
\end{example}

\begin{remark}
Gauss \cite{Mordell,Gauss} showed that the general binary quadratic Diophantine equation can be reduced to a special case of 
the Pell equation. In particular, (\ref{ellipse_3}) can be reduced to
\begin{equation}
U^2-(b^2-4 a c) V^2 = 4 a (a n^2+b n+c)
\end{equation}
where $U=(b^2-4 a c)Y+(b+2 a n)$ and $V=2 a X+b Y+1$ provided that $b^2-4 a c \not= 0$. The trivial factorization 
$F(n) = 1 \cdot F(n)$ corresponds to the solution $U=\pm(2 a n + b)$ and $V=\pm1$.
\end{remark}

\section{Acknowledgements}
I would like to thank John Quintanilla and Nata\u{s}a Jonoska for their useful discussions.

\end{document}